\documentclass[12pt]{amsart}

\usepackage{graphicx, verbatim,amsmath,amssymb}

\newcommand{\bra}{\langle}
\newcommand{\ket}{\rangle}
\newcommand{\www}{{\mathcal W}}
\newcommand{\cut}{\hbox{\hskip 1pt\vrule width3pt height2pt depth1pt \hskip1pt}}

\newcommand{\dg}{{\delta g}}

\newtheorem{thm}{Theorem}[section]
\newtheorem*{thm*}{Theorem}

\newtheorem{lem}[thm]{Lemma}

\everymath{\displaystyle}

\theoremstyle{definition}

\begin{document}
\title{Singularities in the entropy of asymptotically large simple graphs}
\author{Charles Radin and Lorenzo Sadun}

\address{Charles Radin\\Department of Mathematics\\The University of
  Texas at Austin\\ Austin, TX 78712} \email{radin@math.utexas.edu}
\address{Lorenzo Sadun\\Department of Mathematics\\The University of
  Texas at Austin\\ Austin, TX 78712} \email{sadun@math.utexas.edu}

\thanks{This work was partially supported by NSF
  grants DMS-1208941 and DMS-1101326} 

\begin{abstract}
We prove that the asymptotic entropy of large simple graphs, as a
function of fixed edge and triangle densities, is nondifferentiable
along a certain curve.
\end{abstract}

\maketitle

\setlength{\baselineskip}{.6cm}

\section{Introduction}

Extremal graph theory \cite{Bo} deals with graphs in which conflicting
graph invariants are on the verge of contradiction. A classic example
due to Mantel from 1907 shows that, among graphs of order $n$, as the
edge number increases beyond $\lfloor n^2/4 \rfloor$ a graph can no
longer be bipartite and must contain a triangle. Generalizing
slightly, the Mantel problem is to determine those graphs with
fixed edge density $e$ which minimize, and those which maximize, the
possible values $t$ of triangle density. In this vein extremal graph
theory is concerned with qualitative features of graphs with
invariants on the boundary $\partial S$ of the space $S$ of possible
values of some particular set of invariants which, for the Mantel
problem, are the edge and triangle densities, $e$ and $t$. (The set
$S$ for the Mantel problem was finally determined in \cite{Ra}, and
the optimizing graphs in \cite{PR}.) In this paper we are concerned
with a natural generalization of extremal graph theory to the interior
of $S$. Borrowing an idea from physics, it is possible that
qualitative graph features which are forced in an absolute sense on a
subset $P$ of the boundary of $S$ are still retained for typical
graphs in some phase, a region of $S$ abutting $P$. (We define
`phase' below and `typical' in the next section.)  For instance for
the Mantel problem there is evidence in \cite{RS} that for edge
density less than 1/2 there is a region of $S$ abutting the interval
$(e,t) \in [0,1/2]\times\{0\}$ of $\partial S$, in which now a typical
graph is nearly bipartite. (The vertices are divided into two clusters
of nearly equal size, with nearly all edges connecting vertices in one cluster to vertices in the other.) 
One objective in such a study is `phase
transitions', boundaries between phases in which the competition
between invariants which has traditionally been studied on $\partial
S$ is extended into the interior of $S$, and now concerns
typical graphs. We study typical graphs using entropy and the graph
limit formalism, which we sketch after the following summary of
results.

Consider the set $\hat G^n$ of simple graphs $G$ with set $V(G)$ of
(labeled) vertices, edge set $E(G)$ and triangle set $T(G)$, where
the cardinality $|V(G)|=n$. (`Simple' means the edges are undirected
and there are no multiple edges or loops.) We will be concerned with
the asymptotics of $\hat G^n$ as $n$ diverges, specifically in the
relative number of graphs as a function of the cardinalities $|E(G)|$
and $|T(G)|$.

Let $\displaystyle Z^{n,\alpha}_{e,t}$ be the number of graphs in
$\hat G^n$ such that the edge and triangle densities, $e(g)$ and
$t(g)$, satisfy:
\begin{equation} e(G)\equiv \frac{|E(G)|}{{n \choose 2}}
\in (e-\alpha,e+\alpha) \quad \hbox{ and } \quad
t(G)\equiv \frac{|T(G)|}{{n\choose 3}} \in (t-\alpha,t+\alpha).
\end{equation} 
Graphs $g$ in $\displaystyle \cup_{n\ge 1}\hat G^n$ are known to have edge
and triangle densities, $(e(g),t(g))$, whose accumulation points form
a compact subset $R$ of the
$(e,t)$-plane bounded by three curves, $c_1: (e,e^{3/2}), \ \ 0\le e\le
1$, the line segment $l_1: \ (e,0), \ \ 0\le e\le 1/2$, and a certain scalloped
curve $(e,h(e)),\ \ 1/2\le e\le 1$, lying above the curve
$(e,e(2e-1)), \ \ 1/2\le e\le 1$, and meeting it when
$e=e_k=k/(k+1),\ \ k\ge 1$; see \cite{Ra, PR} and references therein, and Figure 1.
(Note the minor shift in emphasis from $S$, as discussed earlier, to
the accumulation points $R$ of $S$.)

\begin{figure}[h]
\vskip.3truein
\includegraphics[width=3in]{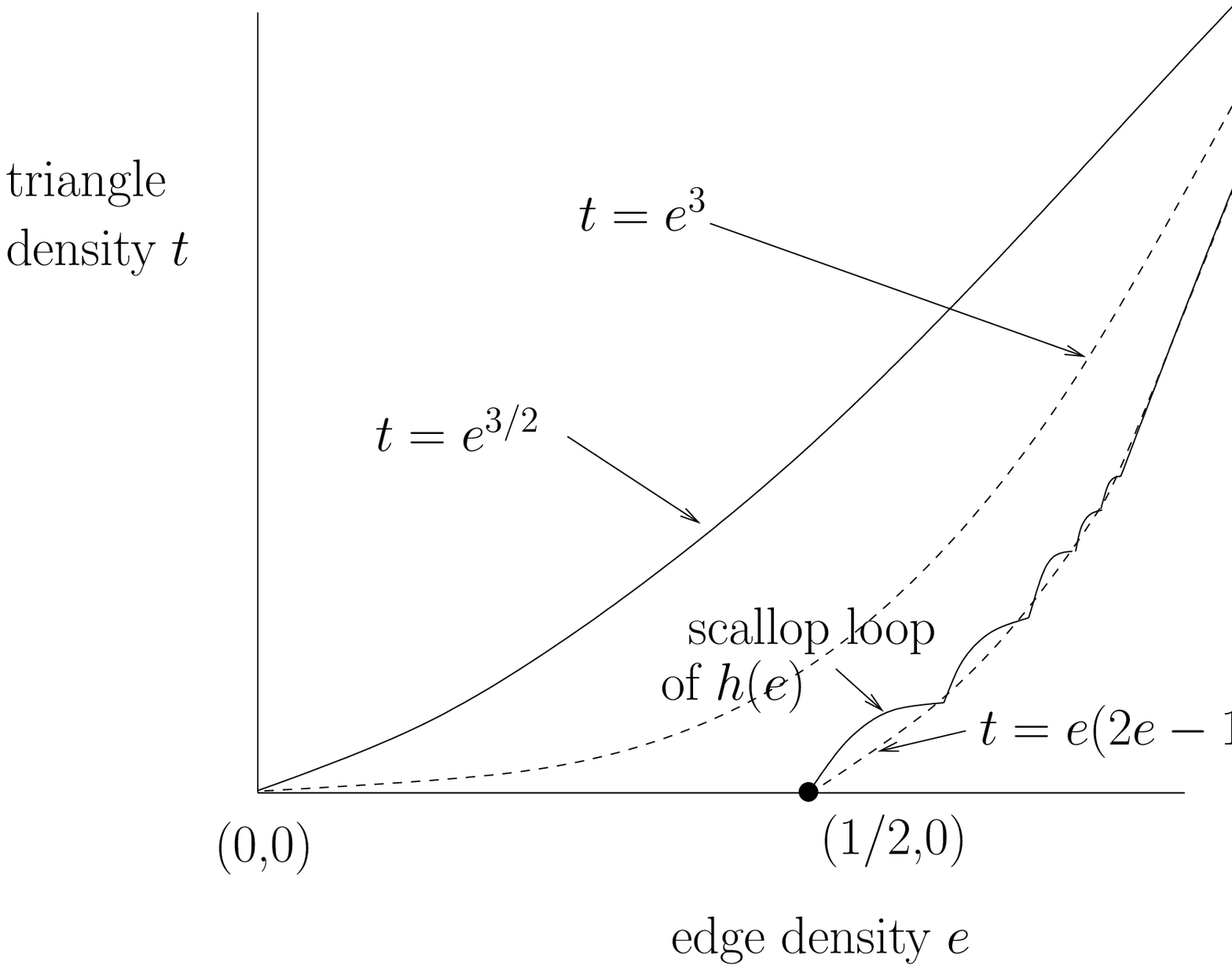}
\caption{The phase space $R$, outlined 
in solid lines}
\label{phasefig}
\end{figure}

We are interested in the relative number of graphs with given
numbers of edges and triangles, asymptotically in the number of
vertices. More precisely we will analyze 
the entropy density, the exponential rate of growth of 
$Z^{n,\alpha}_{e,t}$ as a function of $n$. First consider
\begin{equation} \label{entropydef} 
s^{n,\alpha}_{e,t}=\frac{\ln(Z^{n,\alpha}_{e,t})}{n^2}, \hbox{ and } 
s(e,t)=\lim_{\alpha\downarrow 0}\lim_{n\to \infty}s^{n,\alpha}_{e,t}.
\end{equation} 
The limits defining the entropy density $s(e,t)$ are proven to exist in \cite{RS}.
The  objects of interest for us are the qualitative features of $s(e,t)$ in
the interior of $R$. In particular, a phase is commonly defined as a maximal
connected open subset in which the entropy density is analytic \cite{RY}.
Our main result is:
\begin{thm}\label{thm0}
In the interior of its domain $R$ the entropy density $s(e,t)$ satisfies:
\begin{equation}\label{diff}
s(e,e^3)-s(e,t)\ge c|t-e^3|
\end{equation}
for some $c=c(e)>0$.
Therefore for fixed $e$, $s(e,t)$ attains its maximum at $t=e^3$ but
is not differentiable there. For $t<e^3$ we have the
stronger inequality
\begin{equation}\label{strong}
s(e,e^3)-s(e,t)\ge \tilde c|t-e^3|^{\frac{2}{3}}.
\end{equation}
for some $\tilde c=\tilde c(e)>0$.
\end{thm}
So the graph of $s(e,t)$ has its maxima, varying $t$ for fixed $e$, on
a sharp crease at the curve $t=e^3,\ 0<e<1$, and is not concave for
$t$ just below $e^3$. The importance of the result lies in the implication from
(\ref{diff}) of the lack of differentiability of $s(e,t)$ on the crease,
and thus the existence of a phase transition, and the implication
from (\ref{strong}) of a lack of concavity of $s(e,t)$, discussed below.
\smallskip

We begin with a quick review of the formalism of graph limits, as recently
developed in \cite{LS1, LS2, BCLSV, BCL, LS3}; see also the recent book
\cite{Lov}.
The main value of this formalism here is that one can use large
deviations on graphs with independent edges \cite{CV} to give an
optimization formula for $s(e,t)$ \cite{RS}.


\section{Graphons}

Consider the set $\www$ of all symmetric, measurable functions 
\begin{equation} g:(x,y)\in [0,1]^2\to g(x,y)\in [0,1].\end{equation} 
Think of each axis as a continuous set of vertices of a graph. For a
graph $G\in \hat G^n$ one associates 
\begin{equation} \label{checkerboard} g^G(x,y) = \begin{cases} 1 &\hbox{if }(\lceil nx \rceil , \lceil ny \rceil
)\hbox{ is an edge of }G\cr 0 & \hbox{otherwise,} \end{cases}
\end{equation} 
where $\lceil y \rceil$ denotes the smallest integer greater than or
equal to $y$. 
For $g\in \www$ and simple graph $H$ we define
\begin{equation} t(H,g)\equiv \int_{[0,1]^\ell} \prod_{(i,j)\in E(H)}g(x_i,x_j)\,dx_1\cdots
   dx_\ell, \end{equation} 
where $\ell = |V(H)|$, and note that for a graph $G$, $t(H,g^G)$ is the 
density of graph homomorphisms $H\to G$:
\begin{equation} \frac{|\hbox {hom}(H,G)|}{|V(G)|^{|V(H)|}}. \end{equation} 
We define an equivalence relation on $\www$ as follows: $f\sim g$
if and only if $t(H,f)=t(H,g)$ for every simple graph $H$.  Elements
of $\www$ are called ``graphons'', elements of the quotient space $\tilde \www$ are called ``reduced graphons'', and
the class containing $g\in \www$ is denoted $\tilde g$. 
Equivalent functions in $\www$ differ by a change of variables in the
following sense. Let $\Sigma$ be the space of measure-preserving maps
$\sigma: [0,1]\to [0,1]$, and for $f$ in $\www$ and $\sigma\in
\Sigma$, let $f_\sigma(x,y)\equiv f(\sigma(x),\sigma(y))$. Then $f\sim
g$ if and only if there exist $\sigma, \sigma'$ in $\Sigma$ such that
$f_\sigma =g_{\sigma'}$ almost everywhere; see Cor. 2.2 in \cite{BCL}. The space
$\www$  is compact with respect to the `cut metric'
defined as follows.
First, on ${\www}$ define:
\begin{equation}
{d}_{\cut}(f,g)\equiv \sup_{S,T\subseteq [0,1]}\Big| \int_{S\times
  T}[f(x,y)-g(x,y)]\, dxdy\Big|. 
\end{equation}
Then on $\tilde \www$ define the cut metric by:

\begin{equation}
{\tilde d}_{\cut}(\tilde f,\tilde g)\equiv \inf_{\sigma,\sigma'\in \Sigma}
{d}_{\cut}(f_{\sigma},g_{\sigma'}). 
\end{equation}

We will use the fact, which follows easily from Lemma 4.1 in 
\cite{LS1}, that the cut metric is equivalent to the metric
\begin{equation} \delta_{_{\hbox{hom}}}(\tilde f,\tilde g)\equiv \sum_{j\ge 1}
\frac{1}{2^j}|t(H_j,f)-t(H_j,g)|, 
\end{equation} 
where $\{H_j\}$ is a countable set of simple graphs, one from each
graph-equivalence class. 
Also note that
if each vertex of a finite graph is split into the same number of
`twins', each connected to the same vertices, the result stays in the
same equivalence class, so for a convergent sequence $\tilde g^{G_j}$ one may
assume $|V(G_j)|\to \infty$. 

 The following was proven in \cite{RS}.

\begin{thm}\label{thm1} (\cite{RS})
  For any possible pair $(e,t)$, $s(e,t) = \max [-I(g)]$, where the
  maximum is over all graphons $g$ with $e(g)=e$ and $t(g)=t$,
  where 
\begin{equation} e(g)=\int_{[0,1]^2}  g(x,y) \, dxdy, \qquad t(g) =
  \int_{[0,1]^3} g(x,y) g(y,z) g(z,x) \, dxdydz
\end{equation}
  and the rate function is 
\begin{equation}
I(g) = \int_{[0,1]^2} I_0(g(x,y)) \,dxdy, \hbox{
    \rm where } I_0(u)= \frac{1}{2} \left [u \ln(u) +
    (1-u)\ln(1-u)\right ].
\end{equation} 
\end{thm}

\section{Proof of Theorem \ref{thm0}}

\begin{proof}
Fix a graphon $g$ with edge density $e$. We can 
always write such a graphon as $g = g_e + \dg$ where $g_a$
is the constant function on $[0,1]^2$ with value $a$. We then compute

\begin{eqnarray}
\delta t(g) &:=& t(g) - e^3 = 3e^2\int_{[0,1]^2} \dg(x,y)\,dxdy + 
3e\int_{[0,1]^3} \dg(x,y) \dg(y,z)\,dxdydz\cr
&&+ \int_{[0,1]^3} \dg(x,y) \dg(y,z) \dg(z,x)\, dxdydz.
\end{eqnarray}
The first term on the right hand side is zero, since $\int_{[0,1]^2}
\delta g(x,y)\, dxdy = \delta e = 0$.  If we think of $\dg$ as the
integral kernel of the Hermitian trace class operator $T_\dg$ on
$L^2([0,1])$, then using the inner product $\langle\, \cdot \, ,\, \cdot\, \rangle$ and
trace $Tr$
we can rewrite the remaining terms as
\begin{equation}
\delta t = 3 e\langle \phi_1, T_\dg^2 \phi_1 \rangle + Tr(T_\dg^3),
\end{equation}
where $\phi_1(x)=1$ is the constant function on $[0,1]$.
Note that the first
term is non-negative.  Using again the fact that $\int_{[0,1]^2} \dg(x,y)\,dxdy=0$,
\begin{eqnarray}\label{Ibelow}
\delta I &= & \int_{[0,1]^2} [I_0(e+\dg(x,y))-\dg(x,y) I_0'(e) -I_0(e)]\,dxdy \cr  & =& 
\int_{[0,1]^2} \frac{I_0(e+\dg(x,y))-\dg(x,y) I_0'(e) -I_0(e)}{\dg(x,y)^2} \dg(x,y)^2\,dxdy \cr
&\ge& f_-(e) \int_{[0,1]^2} \dg(x,y)^2\,dxdy, 
\end{eqnarray}
where
\begin{equation}\label{positive}
f(e,x) = \frac{I_0(e+x)-xI_0'(e)-I_0(e)}{x^2},
\end{equation}
and
$f_-(e) = \inf_x f(e,x)$
is a positive number less than or equal to $\frac{I_0''(e)}{2}=\frac{1}{4e(1-e)}$.

\begin{lem}\label{lem} $|Tr(T_\dg^3)| \le (Tr(T_\dg^2))^{3/2}$, with equality
if and only if $T_\dg$ is a rank 1 operator. \end{lem}

\begin{proof} Since $T_\dg$ is an Hermitian trace class operator it has pure
discrete spectrum. If $\{ \mu_i\}$ are the eigenvalues of $T_\dg$, then
\begin{equation}
|Tr(T_\dg^3)| = |\sum_i \mu_i^3| \le \sum_i |\mu_i^3| \le
\max_j|\mu_j| \sum_i \mu_i^2 \le (\sum_i \mu_i^2)^{3/2} = [Tr(T_\dg^2)]^{3/2}.
\end{equation}
If $T_\dg$ has rank one, then $Tr(T_\dg^3)=\mu^3 = \pm [Tr(T_\dg^2)]^{3/2}$. If 
$T_\dg$ has rank bigger than 1, then $\max_j(\mu_j)$ is strictly smaller
than $\sqrt{\sum_i \mu_i^2}$. 
\end{proof}


We next give an estimate for $I(g)$ when $t<e^3$.
If $\delta t<0$, then 
\begin{equation} -\delta t = -Tr(T_\dg^3) - 3 e \langle \phi_1, T_\dg^2\phi_1 \rangle
\le -Tr(T_\dg^3) \le [Tr(T_\dg^2)]^{3/2} \le \left (\frac{\delta I}{f_-(e)}\right)
^{3/2}.
\end{equation}
This implies that 
\begin{equation} \label{deltat23}
\delta I \ge f_-(e) (-\delta t)^{2/3}.
\end{equation}
Using $|\delta t| \le e^3$ this also implies a linear estimate
\begin{equation}
\delta I \ge \frac{f_-(e)}{e} |\delta t|
\end{equation}
for $\delta t < 0$. 

Finally, we estimate $I(g)$ when $t>e^3$. Since $\langle \phi_1, T_\dg^2 \phi_1 \rangle \le Tr(T_\dg^2)$, 
and since $Tr(T_\dg^2) \le 1$, we have
\begin{equation}\label{deltat1}
\delta t \le Tr(T_\dg^3) + 3e Tr(T_\dg^2) \le (Tr(T_\dg)^2)^{3/2} + 3e Tr(T_\dg^2) \le (3e+1) Tr(T_\dg^2)
\le \frac{(3e+1) \delta I}{f_-(e)},
\end{equation}
so
\begin{equation} \delta I \ge \frac{f_-(e) \delta
    t}{3e+1}. 
\end{equation}
\end{proof}

\section{Other graph models}

We now generalize Theorem \ref{thm0} to graph models where we keep track of
the number of graph homomorphisms $H \to G$ for some fixed graph $H$,
not necessarily triangles.  We can compute the entropy of graphs with $e(g^G)$ within
$\alpha$ of $e$ and $t(H, g^G)$ within $\alpha$ of $t$, and define the entropies
$s^{n,\alpha}_{e,t}$ and $s(e,t)$ exactly as in equation (\ref{entropydef}).
The proof of Theorem \ref{thm1} carries over almost word-for-word to
show the following.
\begin{thm} \label{thm3}
 For any possible pair $(e,t)$, $s(e,t) =  \max [-I(g)]$, where the
  maximum is over all graphons $g$ with $e(g)=e$ and $t(H,g)=t$.
\end{thm}

Note that if $H$ has $k$ edges the constant graphon
$g_e$ satisfies $t(H,g_e)=e^k$.

\begin{thm}\label{thm4}
For fixed $0<e<1$ the entropy density $s(e,t)$ achieves its maximum
at $t=e^k$ and is not differentiable 
with respect to $t$ at that point.
\end{thm}

\begin{proof} Following the proof of Theorem \ref{thm0}, we write 
$g = g_e + \delta g$ and expand both $I(g)$ and $t(H,g)$ in terms
of $\delta g$. The estimate (\ref{Ibelow}) still 
applies.  The only difference is the expansion of $t(H,g)$. 

Since $t(H,g)$ is the integral of a polynomial expression in $g$, we can expand
$\delta t$ as a polynomial in $\delta g$.  This must take the form
\begin{eqnarray} 
\delta t &=& \int_{[0,1]^2} h_1(x,y) \delta g(x,y)\, dx dy + 
\int_{[0,1]^4} h_2(w,x,y,z) \delta g(w,x) \delta g(y,z)\, dw dx dy dz \cr
&& + \int_{[0,1]^3} h_3(x,y,z) \delta g(x,y) \delta g(y,z)\, dxdydz + \cdots,
\end{eqnarray}
where the non-negative functions $h_1(x,y)$, $h_2(w,x,y,z)$,
$h_3(x,y,z)$, etc., are computed from the graphon from which we are
perturbing. However, that graphon is a constant $g_e$, so each
function $h_i$ is also a constant.  Thus there are non-negative
constants $c_1$, $c_2$, $\ldots$, such that
\begin{eqnarray} \label{bigsum}
\delta t = c_1 \int_{[0,1]^2} \delta g(x,y) \, dxdy\ &+& c_2 \int_{[0,1]^4} \delta g(w,x)\delta
g(y,z)\, dw dx dy dz \cr 
&+& c_3 \int_{[0,1]^3} \delta g(x,y) \delta g(y,z)\, dx dy dz + \cdots
\end{eqnarray}
The first two terms integrate to zero, while any subsequent terms are 
bounded by a multiple of $Tr(T_{\delta g}^2)$. Since there are only a finite 
number of terms, $|\delta t|$ is bounded above by a constant multiple of 
$Tr(T_{\delta g}^2)$ while $\delta I$ is bounded below by a constant 
multiple of $Tr(T_{\delta g}^2)$. Combining these observations yields
the analog of Theorem \ref{thm0}, and we conclude that 
$s(e,t)$ cannot have a 2-sided derivative with respect to $t$ at $t=e^k$.
\end{proof}
 
A more careful analysis of the terms in the sum (\ref{bigsum}) shows that each term is either
positive-definite, is dominated by a positive-definite term, or scales as $Tr(T_{\delta g}^2)^{3/2}$ or higher, implying 
that the concavity of $s(e,t)$ just below the curve $t=e^k$ is the same as for
the triangle model. However, this analysis is not needed for the proof of Theorem 
\ref{thm4} and has been omitted. 

There do exist some graphs $H$, such as ``$k$-stars'' with $k$
edges and one vertex on all of them, such that the lowest value of $t$
for fixed $e$ is on the `Erd\H{o}s-R\'enyi curve', $t=e^k,\ 0<e<1$.
For such graphs the analysis of what happens for $\delta t<0$
is moot and $s(e,t)$ may have a 1-sided derivative at $(e,e^k)$. 
 
\section{Legendre transform and exponential random graphs}

We return temporarily to the special case in which $H$ is a triangle.
Note that it
has been fundamental to our analysis to use the optimization
characterization of $s(e,t)$ of Theorem \ref{thm1} (Theorem 3.1 in \cite{RS}). 
Treating this
as an optimization with constraints, one might naturally introduce
Lagrange multipliers $\beta_{1}, \beta_2$ and consider the following 
optimization system,
\begin{equation}
\max_g [-I(g)+\beta_1 e(g) + \beta_2 t(g)];\ e(g)=e;\ t(g)=t,
\end{equation}
namely maximize
\begin{equation}\label{maxwithbeta}
\Psi_{\beta_1,\beta_2}(g)=-I(g)+\beta_1 e(g)+\beta_2 t(g)
\end{equation}
for fixed $(\beta_1,\beta_2)$ and then adjust $(\beta_1,\beta_2)$ to
achieve the desired values of $e(g)$ and $t(g)$. 
The `free energy density' 
\begin{equation}
\psi(\beta_1,\beta_2)=\max_g \Psi_{\beta_1,\beta_2}(g)
\end{equation}
is directly related to the normalization in exponential random graph
models and basic information in such models is simply obtainable from
it \cite{N, CD, RY, AR}. It can be considered the Legendre transform
of $s(e,t)$, but since the domain of $s(e,t)$ is not convex, the
relationship between $s(e,t)$ and $\psi(\beta_1,\beta_2)$ must be more
complicated than is common for Legendre transforms. 
In particular, although it has been proven
(\cite{CD, RY}) that $\psi(\beta_1,\beta_2)$ has singularities as a
function of $(\beta_1,\beta_2)$ (see Figure 2) it does not seem
straightforward to use this to prove singularities in $s(e,t)$. This is what
necessitated the different approach we have taken
here. We will try to clarify the relationship between
$\psi(\beta_1,\beta_2)$ and $s(e,t)$ through differences in the
optimization characterizations of these quantities.

\begin{figure}[h]
\vskip.4truein
\includegraphics[width=3in]{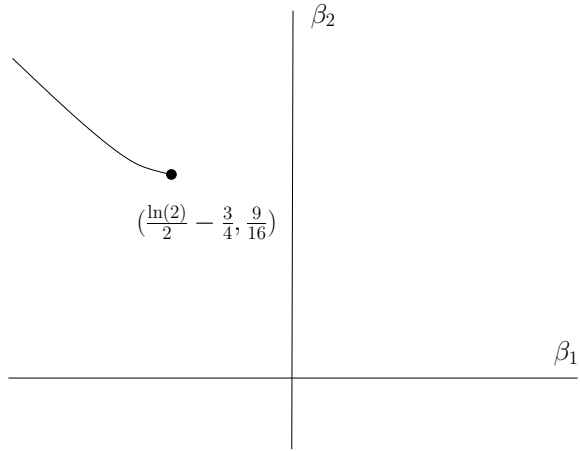}
\caption{The curve of all singularities of $\psi(\beta_1,\beta_2)$,
  for 
$\beta_2>-1/2$}
\label{trans}
\end{figure}

As one crosses the curve in Figure 2 by increasing $\beta_2$ at fixed
$\beta_1$, the unique graphon maximizing 
$\Psi_{\beta_1,\beta_2}(g)$ jumps from lower to higher value of
$e(g)$, but still $t(g)=e(g)^3$ \cite{CD,RY}. We emphasize that whenever
$\beta_2>-1/2$, one is on the Erd\H{o}s-R\'enyi curve $t=e^3$
indicated in Figure 1 \cite{CD,RY}. This is significant in interpreting the
singularities of $s(e,t)$ and $\psi(\beta_1,\beta_2)$.  The
singularities or `transition' characterized in Theorem \ref{thm0} and
associated with {\it crossing} the Erd\H{o}s-R\'enyi curve is presumably
between graphs of different character but similar densities; we expect
that graphons maximizing $s(e,t)$, for $t> e^3$, are related to those
(discussed below)
for the upper boundary of its domain $R$, while for $t< e^3$ they are
related to those for the lower boundary of $R$. (The latter are the
subject of \cite{RS, AR}.) On the other hand, the transition in Figure 2,
associated with varying $(\beta_1,\beta_2)$,
is between graphs of similar character (independent edges) but
different densities. This phenomenon is unrelated to the transition of
Theorem \ref{thm0}, although still associated with the Erd\H{o}s-R\'enyi
curve, and which we understand as follows.

Assume one optimizes $\Psi_{\beta_1,\beta_2}(g)$ for fixed
$(\beta_1,\beta_2)$, where $(\beta_1,\beta_2)$ is adjusted so that
maximizing graphons $g$ satisfy $e(g)=e$ and $t(g)=t$ to match the
desired values of $(e,t)$ in which we are interested for $s(e,t)$. It
may happen that for special $(\beta_1,\beta_2)$ there are also
optimizing $g$ with other densities, $(e(g),t(g))\ne
(e,t)$. This degeneracy is what is occurring precisely for the
$(\beta_1,\beta_2)$ on the singularity curve of Figure 2. All
such $g$ clearly solve the maximization problem for $s[e(g),t(g)]$;
they are appearing together when we fix $(\beta_1,\beta_2)$ because
the value of $\Psi_{\beta_1,\beta_2}(g)$ happens to be the same for all these $g$, 
a phenomenon of no particular relevance to the original
optimization problem of $s(e,t)$. So in this sense degenerate
solutions in the Lagrange multiplier method can be misleading; they
point to a `transition' which is foreign to the maximization problem
for $s(e,t)$. We next
consider other features of the Lagrange multiplier method.

One issue of importance to those who study exponential random graph
models is that for no $\beta_2$ is there a maximizer $g$ of the free
energy density $\Psi$ satisfying $t(g) > e(g)^3$, though there clearly are
such optimizers of the entropy density $s$ as we see for instance from
Figure 1.

\begin{thm}\label{thm5}
For every $\beta_2$ and every maximizer $g$ of $\Psi(g)$, $t(g) \le e(g)^3$.
\end{thm}

\begin{proof} Suppose the graphon $g^\prime$ satisfies 
$t(g^\prime) > [e(g^\prime)]^3$ and maximizes the free energy

\begin{equation}  \Psi(g) = -I(g) + \beta_1 e(g) + \beta_2 t(g),
\end{equation} 
for some $\beta_1$ and $\beta_2$. It follows from Theorem 4.2 in \cite{CD} 
that $\beta_2 < 0$.  Let $g_e$ be the constant graphon with the same edge density
as $G^\prime$. Since $t>e^3$, $\beta_2(g_e) > \beta_2 t(g^\prime)$. Also, $-I(g_e) > -I(g^\prime)$,
since for given edge density $-I(g)$ is maximized at $g_e$. But then $\Psi(g_e) > \Psi(g^\prime)$, and $g^\prime$ is
not a maximizer, which is a contradition. 
\end{proof}

\section{Optimizing graphons}

Having established in Theorem \ref{thm0} a phase transition on the Erd\H{o}s-R\'enyi curve, we consider the forms of the graphons
that maximize $s(e,t)$ on each side of the curve. We previously \cite{RS} determined the optimizing graphons
on the lower boundary of the region $R$, including the scalloped curve. We now compute the optimizing graphons on 
the upper boundary and on the curve $e=1/2$ below the Erd\H{o}s-R\'enyi line.   

\subsection{The upper boundary}

\begin{thm}\label{thm6}
If $g$ maximizes $s(e,e^{3/2})$ it takes the form
\begin{equation} \label{top} g(x,y) = \begin{cases} 1 & x,y < \sqrt{e} \cr 0 
& \hbox{otherwise.}\end{cases}
\end{equation}
up to a measure-preserving transformation.
\end{thm}

\begin{proof}

Let $T_g$ be the operator on $L^2[0,1]$ with integral kernel $g$. We
already know that $t = Tr(T_g^3) \le Tr(T_g^2)^{3/2}$, with equality
if and only if $T_g$ is rank 1. However,
\begin{equation}
Tr(T_g^2) = \iint g(x,y) g(y,x) dx dy = \iint g(x,y)^2 dx dy \le
\iint g(x,y) dx dy = e,
\end{equation}
 with equality if and only if $g(x,y)^2 =
g(x,y)$ almost everywhere, i.e. $g(x,y) = 0$ or 1 almost everywhere.

Combining the two results, we have that $t \le e^{3/2}$, with equality
if and only if two conditions are met: $g(x,y) = \alpha(x) \alpha(y)$
for some positive function $\alpha$, (i.e. $T_g$ has rank one), and
$g(x,y)$ is a 0--1 function, implying that $\alpha(x)$ is a 0-1
function.

By applying a measure-preserving transformation to $[0,1]$ we can
assume that $\alpha$ is the characteristic function of an interval
$[0,s]$. We then compute $e=s^2$ and $t=e^3$. In short, each point on
the upper boundary for the allowed region in the $e$--$t$ plane is
achieved by a unique reduced graphon, namely the equivalence class of
the graphon (\ref{top}).
\end{proof}

\subsection{The special case of $e=1/2$}

\begin{thm} When $e=1/2$ and $t \le e^3$, the graphon
\begin{equation}\label{minimizer}
\tilde g(x,y) = \begin{cases} 1/2 + \epsilon  & x < 1/2 < y \hbox{
    or } x > 1/2 > y \cr 1/2 - \epsilon & x,y < 1/2 \hbox{ or } x,y
  > 1/2, \end{cases}
\end{equation}
where $\epsilon = (e^3-t)^{1/3}$, maximizes $s(e,t)$. Furthermore,
every maximizing graphon is of the form $\tilde g_\sigma$ for some 
mearsure-preserving transformation $\sigma$. 
\end{thm}

\begin{proof} 

We use perturbation theory, writing $g(x,y) = e + \delta g(x,y)$. 
When $e=1/2$, the $n^{th}$ derivative $I_0^{(n)}(x)$ is positive for $n$
even and zero for $n$ odd.  This means that
$[{I_0(e+x)-I_0(e)}]/{x^2}$ is a convex function of $x^2$ (since it is
a power series in $x^2$ with positive coefficients).  This allows us
to find a formula for $\delta g$ that simultaneously maximizes
$-Tr(T_\dg^3)$ for fixed $Tr(T_\dg^2)$, minimizes the
positive-definite quadratic term in $\delta t$ (to be zero), and
minimizes $\delta I$ for fixed $Tr(T_\dg^2)$. This must therefore be a
minimizer of the rate function and a maximizer of the entropy. We assume throughout
that $\iint \delta g(x,y) dx\, dy = 0$. 

\begin{lem} Let $T_\dg$ be a rank-one operator: $T_\dg f=c \langle \alpha,
  f \rangle \, \alpha$ where $\langle \alpha, \alpha\rangle = 1$. 
Then $\delta t = c^3$.
\end{lem}

\begin{proof}  Since $c \bra g_1, \alpha \ket \bra \alpha, g_1 \ket = \int_{[0,1]^2} \dg(x,y) = 0$,
we must have $\bra g_1, \alpha \ket =0$. This makes the quadratic term 
$3e \bra g_1,T_\dg^2 g_1 \ket$ identically zero. Since $T_\dg$ is 
rank one with unique eigenvalue $c$, $\delta t = Tr(T_\dg^3) = c^3$.
\end{proof}

Now we try to minimize $\int_{[0,1]} I[e+c\alpha(x)\alpha(y)]\, dxdy$. By convexity,
this is minimized when $[\alpha(x)\alpha(y)]^2$ is constant, which
means that $\alpha(x)^2$ is constant. Since the integral of $\alpha$
is zero, we must have $\alpha(x)=+1$ on a set of measure 1/2 and $-1$
on a set of measure 1/2.  Up to measure-preserving automorphism, we
can assume that
\begin{equation}
\alpha(x) = \begin{cases} 1 & x>1/2; \cr -1 & x < 1/2. \end{cases}
\end{equation}
This means that any graphon that minimizes $I(g)$ for fixed $e=1/2$
and fixed $t\le e^3$ must be $\tilde g$, up to a measure-preserving transformation.  
\end{proof}

\subsection{Lagrange multipliers on the $e=1/2$ line}

We proved in Theorem \ref{thm5} that maximizing graphons for $s(e,t)$
for $t > e^3$ cannot be found using Lagrange multipliers. We now show
that this also applies
to certain values of $t<e^3$, starting with $e=1/2$ and $t$ close to $1/8$. 

When $e=1/2$, knowing precisely the optimizing graphon $\tilde g$ allows us to compute $s(e,t)$:
\begin{equation} \label{halfentropy} s(\frac12,t) = \frac{-1}{2}\left [I_0\left (\frac12 + \epsilon\right )+I_0\left (\frac12 - \epsilon\right )\right ] =
-I_0\left (\frac12 + \epsilon\right ),
\end{equation}
for all $t<1/8$, since $I_0(u)=I_0(1-u)$. 

 Now consider the optimization using Lagrange multipliers. The
 Euler-Lagrange equations are:
\begin{equation}\label{E-L}
-I_0'[g(x,y)]  + \beta_1 + \beta_2 h(x,y)=0,
\end{equation}
where 
\begin{equation}\label{h}
h(x,y) = 3 \int_{[0,1]} g(x,z) g(y,z)\, dz
\end{equation}
is the first variation of $t(g)$ with respect to $g(x,y)$. For our
$g=\tilde g$,  this becomes:
\begin{eqnarray}
\beta_1 + 3 \beta_2\left(\frac{1}{4}-\epsilon^2\right) & = & I_0'\left (\frac{1}{2}+\epsilon\right ) 
= \frac{1}{2}\ln\left[\frac{\frac{1}{2}+\epsilon}{\frac{1}{2}-\epsilon}\right] \cr
\beta_1 + 3 \beta_2\left ( \frac{1}{4}+\epsilon^2\right ) & = & I_0'\left (\frac{1}{2}-\epsilon\right) 
= \frac{1}{2}\ln\left[\frac{\frac{1}{2}-\epsilon}{\frac{1}{2}+\epsilon}\right],
\end{eqnarray}
which are satisfied if and only if
\begin{equation}\label{betaEL}
\beta_2 = - \frac{4}{3}\beta_ 1= \frac{I_0'\left (\frac{1}{2}-\epsilon\right) -
  I_0'\left(\frac{1}{2}+\epsilon\right)}{6\epsilon^2}=
-\frac{1}{6\epsilon^2}\ln \left[ \frac{\frac{1}{2} +\epsilon}
{\frac{1}{2}-\epsilon}\right].
\end{equation} 
Notice that $\beta_1$ and $\beta_2$ diverge as 
$\epsilon \downarrow 0$ (equivalently, as $t \uparrow 1/8$). 

However, solutions to the Euler-Lagrange equations are not necessarily local maxima of $\Psi$. 
It is easy
to check by differentiation of (\ref{halfentropy}) that there are $0 < c_1 < c_2 < 1/8$ such
that $s(1/2,t)$ is strictly concave on $(0,c_1)$ but strictly convex
on $(c_2,1/8]$. Convexity implies that $\tilde g$ is not a maximizer
  for $\Psi(\beta_1,\beta_2)$ for $c_2<t< 1/8$, but is rather a local \em{minimizer}
  with respect to variation of $t$, and so there are no
  $(\beta_1,\beta_2)$ which can lead to the maximizers of $s (1/2,t)$
  for $t$ just below $1/8$. While the precise calculation was done for $e=1/2$ using 
equation (\ref{halfentropy}), this phenomenon is simply due to
  inequality (\ref{strong}), and actually occurs for all $e$, not
  just for $e=1/2$.  In fact from the proof of Theorem \ref{thm4} this
  phenomenon can be extended to subgraphs $H$ other than triangles.
  [However, as noted above, for some $H$ the Erd\H{o}s-R\'enyi curve is
  actually the lower boundary of the domain of the entropy, in which case there
  are no `missing' points below it.]
\vskip.2truein
\noindent {\bf Acknowledgements:} We gratefully acknowledge useful 
discussions with Mei Yin.

\end{document}